 \documentclass[11pt]{amsart}
\usepackage{amsfonts}
\usepackage{amssymb}
\usepackage{amsmath}
\usepackage{amsthm}

\normalsize

\newtheorem{lemma}{Lemma}
\newtheorem{theorem}{Theorem}

\theoremstyle{definition}
\newtheorem{definition}{Definition}
\newtheorem{corollary}{Corollary}
\newcommand{\te}{Teich\-m\"ul\-ler}
\newcommand{\tes}{Teichm\"uller's}

\theoremstyle{definition}

\newcommand{\cir}{\mathbb{ S}^1}

\begin{document}

\title[UAA and UAC endomorphisms]{Asymptotically Affine
and Asymptotically Conformal Circle Endomorphisms}

\author{Frederick P. Gardiner and Yunping Jiang}

\address{Frederick P. Gardiner: Department of Mathematics\\
2900 Bedford Avenue\\
Brooklyn, NY 11210-2889\\
and\\
Department of Mathematics\\
Graduate School of the City University of New York\\
365 Fifth Avenue, New York, NY 10016}
\email{frederick.gardiner@gmail.com}

\address{Yunping Jiang: Department of Mathematics\\
Queens College of the City University of New York\\
Flushing, NY 11367-1597\\
and\\
Department of Mathematics\\
Graduate School of the City University of New York\\
365 Fifth Avenue, New York, NY 10016}
\email{yunping.jiang@qc.cuny.edu}



\thanks{The research is partially supported by
PSC-CUNY awards.}

\begin{abstract}
We show that every uniformly asymptotically affine circle
endomorphism has a uniformly asymptotically conformal extension.
\end{abstract}

\maketitle

\section*{Introduction}
   First we summarize basic properties of uniformly asymptotically affine
   circle degree $d>1$ endomorphisms.  Then we use
   the Beurling-Ahlfors extension to realize any uniformly
   asymptotically affine system as the restriction to the circle of a
   uniformly asymptotically conformal system. Theorem 1 is a
   well-known
   characterization of symmetric homeomorphisms of the real axis in
   terms of possible quasiconformal extensions. Theorem~\ref{Cui}
   is an exposition of calculations given by Cui in~\cite{Cui}.
   Theorem~\ref{qsconjugacy} is a special case of a theorem for
   any one-dimensional Markov maps with bounded geometry in~\cite{Jiang0,Jiang1} (see also~\cite{Jiang}).
   Theorem~\ref{uaatouac} is the main new result. Since it is known that the
    \te\
    space of uniformly asymptotically affine expanding maps is
    complete with \tes\ metric (see, for example~\cite{GardinerSullivan,CuiGardinerJiang,Jiang}),
    this theorem shows that the
    uniformly asymptotically conformal expanding maps also form a
    complete metric space with \tes\ metric.  In a subsequent paper we will exploit this fact to
    construct a dual dynamical system corresponding to every
    UAA circle expanding map.

    We take this opportunity to express our gratitude to the
    referee for several important and helpful corrections.

\section{Circle endomorphisms}
Let $\cir=\{ z\in {\Bbb C}\;  : \; |z|=1\}$ be the unit circle.
The map $\pi: {\Bbb R} \rightarrow \cir$ defined by
$$\pi (x) = e^{2\pi ix}$$  realizes ${\mathbb R}$ as
the universal covering of $\cir$
  with covering $\pi$ and covering group
${\Bbb Z}.$   $\pi$ induces an isomorphism from ${\Bbb R}/{\Bbb
Z}$ onto $\cir.$

Let $m$ be the degree of an orientation preserving covering $f$
from $\cir$ onto itself and assume $1<m<\infty.$ $f$ is an
endomorphism of $\cir$ and it necessarily has one fixed point $p.$
By selecting an orientation preserving M\"obius transformation $A$
that preserves the unit disk with $A(p)=1,$ we may shift
consideration of the map $f$ to the map $\tilde{f}=A \circ f \circ
A^{-1}.$  $\tilde{f}$ has the same dynamical properties as $f$ and
it fixes the point $1.$  Therefore, without loss of generality, we
may assume to begin with that $f$ fixes the point $p=1.$  We
denote the homeomorphic lift of $f$ by $F.$ $F$ is uniquely
determined by $f$ if we assume it has the following properties:

i) $F$ is a homeomorphism of ${\mathbb R},$

ii) $\pi \circ F = f \circ \pi,$

iii) $F(0)=0.$

Note that $F(x+1)=F(x)+m.$  In this paper  we refer either to $f$
or to its unique corresponding lift $F$ as a  circle endomorphism.
We denote the $n$-fold composition of $f$ with itself by  $f^{n}.$
Similarly, $F^{n}$ is the $n$-fold composition of $F.$

Suppose $h$ is an orientation-preserving circle homeomorphism. Let
$H$ be the lift of $h$ to ${\mathbb R}$ such that $0\leq H(0)<1$.
Then $H(x+1)=H(x)+1$ and $h$ and $H$ are one-to-one
correspondences.

\begin{definition}
A circle homeomorphism $h$ is called $M$-quasisymmetric if there
is a constant $M\geq 1$ such that for all real numbers $x$ and for
all $y>0,$
\begin{equation}\label{Mcondition}
 \frac{1}{M} \leq \frac{H(x+y)-H(x)}{H(x)-H(x-y)}\leq M.
\end{equation}
\end{definition}

The expression
$$
\rho_{H}(x,y)= \frac{H(x+y)-H(x)}{H(x)-H(x-y)}, \quad x,\; y\neq 0
\in {\mathbb R},
$$
is called the quasisymmetric distortion function for $H$. We also
need the skew quasisymmetric distortion function for $H,$ which is
defined by
$$
\rho_{H}(x,y, k)=\frac{H(x+ky)-H(x)}{H(x)-H(x-y)}, \quad x,\;
y\neq 0\in {\mathbb R},\;\; 0<k\leq 1.
$$

The following lemma is well-known for a quasisymmetric
homeomorphism of the real line. One can prove it by using
quasiconformal mapping theory (see, for
example,~\cite{LehtoBook}). However, we need to use it for a
quasisymmetric homeomorphism of a compact interval. For
quasisymmetric homeomorphisms of compact intervals, the proof by
using quasiconformal mapping theory does not work since a
$M$-quasisymmetric homeomorphism of a compact interval is not
necessarily a restriction of a $M$-quasisymmetric homeomorphism of
the real line. However, there is a proof for a quasisymmetric
homeomorphism of a compact interval by using elementary real
analysis methods (see, for example,~\cite{Jiang}). For the
convenience of the reader we also give the proof here. First we
define the notion of quasisymmetry for a homeomorphism of a
compact interval. An orientation-preserving homeomorphism $H$ of a
closed interval $[a,b]$ is called $M$-quasisymmetric if
$$
M^{-1}\leq \frac{H(x+t)-H(x)}{H(x)-H(x-t)} \leq M, \quad \forall\;
x, x+t, x-t\in [a,b].
$$

\vspace*{10pt}
\begin{lemma}~\label{sd}
There is a function $\zeta (M)>0$ satisfying $\zeta (M) \to 0$ as
$M\to 1$ such that for any $M$-quasisymmetric homeomorphism $H$ of
$[0,1]$ with $H(0)=0$ and $H(1)=1$,
$$
|H(x)-x|\leq \zeta (M), \quad \forall\; x \in [0,1].
$$
\end{lemma}

\begin{proof}
Consider points $x_{n}=1/2^{n}$, $n=0, 1, \cdots$. The
$M$-quasisymmetry condition implies that
$$
M^{-1}\leq
\frac{H(\frac{1}{2^{n-1}})-H(\frac{1}{2^{n}})}{H(\frac{1}{2^{n}})-H(0)}\leq
M.
$$
From this and the fact that $H(0)=0$, we get
$$
(1+M^{-1})H(\frac{1}{2^{n}})\leq H(\frac{1}{2^{n-1}}) \leq
(1+M)H(\frac{1}{2^{n}}).
$$
This gives
$$
\frac{1}{1+M} H(\frac{1}{2^{n-1}}) \leq H(\frac{1}{2^{n}}) \leq
\frac{1}{1+M^{-1}} H(\frac{1}{2^{n-1}}).
$$
Using the fact that $H(1)=1$, we further get
$$
\Big(\frac{1}{1+M}\Big)^{n} \leq H(\frac{1}{2^{n}}) \leq
\Big(\frac{1}{1+M^{-1}}\Big)^{n}, \quad \forall\; n\geq 1.
$$
Furthermore, by $M$-quasisymmetry  and induction on $n=1,
2,\cdots$, yield
$$
\Big(\frac{1}{1+M}\Big)^{n} \leq H(\frac{i}{2^{n}}) -
H(\frac{i-1}{2^{n}}) \leq \Big( \frac{1}{1+M^{-1}}\Big)^{n}, \quad
\forall \; n\geq 1, \;\; 1\leq i\leq 2^{n}.
$$

Let
$$
\tau_n=
\max\left\{\left(\frac{M}{M+1}\right)^n-\frac{1}{2^n},\frac{1}{2^n}-\left(\frac{1}{M+1}\right)^n\right\},
\quad n=1,2, \cdots.
$$
Then for $n=1$,
$$
|H(\frac{1}{2}) -\frac{1}{2}| \leq
\tau_{1}=\frac{1}{2}\frac{M-1}{M+1},
$$
and for any $n>1$, we have
$$
\max_{0\leq i\leq 2^{n}} \Big| H(\frac{i}{2^{n}})
-\frac{i}{2^{n}}\Big| \leq \max_{0\leq i\leq 2^{n-1}} \Big|
H(\frac{i}{2^{n-1}}) -\frac{i}{2^{n-1}}\Big| + \tau_{n}
$$
By summing over $k$ for $1 \leq k \leq n,$ we obtain
$$
\max_{0\leq i\leq 2^{n}} \Big| H(\frac{i}{2^{n}})
-\frac{i}{2^{n}}\Big| \leq \delta_{n}=\sum_{k=1}^{n} \tau_{k}.
$$
If we put $\zeta (M) = \sup_{1\leq n<\infty}\{\delta_{n}\},$ by
summing geometric series, we obtain
$$
\zeta(M) = \max_{1\leq n<\infty} \Big\{
M-1+\frac{1}{2^{n}}-M\Big(\frac{M}{1+M}\Big)^{n},
1-\frac{1}{M}+\frac{1}{M}\Big(\frac{1}{M}\Big)^{n}
-\frac{1}{2^{n}}\Big\}.
$$
Clearly,  $\zeta(M)\to 0$ as $M\to 1$, and since the dyadic points
$$
\{ i/2^{n}\;\; |\;\; n=1, 2, \cdots ; 0\leq i\leq 2^{n}\}
$$
are dense in $[0,1]$, we conclude
$$
|H(x)-x| \leq \zeta (M) \quad \forall \; x\in [0,1],
$$
which proves the lemma.
\end{proof}

\vspace{10pt}
\begin{corollary}~\label{sdc}
Let $\vartheta (M) =M-1 + M\zeta (M)$. Then for any homeomorphism
$H$ of ${\mathbb R}$ and any $x, y>0\in {\mathbb R}$, if $H$
 restricted to the interval $[x-y,x+y]$ is $M$-quasisymmetric, then
$$
\max \left\{|\rho_{H}(x,y, k)-k|,  |\rho_{H}(x,-y, k)-k|
\right\}\leq \vartheta (M),\quad \forall\; 0<k\leq 1.
$$
\end{corollary}

\begin{proof}
Consider $\hat{H}(k) = (H(x+ky) -H(x))/(H(x+y)-H(x))$. Then
$\hat{H}(1)=1$ and  $\hat{H}(0)=0.$ Also, $\hat{H}$ is
quasisymmetric because
$$\frac{\hat{H}(k+j)-\hat{H}(k)}{\hat{H}(k)-\hat{H}(k-j)}=
\frac{{H}(x+ky+jh)-{H}(x+ky)}{{H}(x+ky)-{H}(x+ky-jy)}
$$
for any $0\leq k\leq 1$ and $j>0$ such that $[k-j,k+j]\subset
[0,1]$ and this is bounded above by $M$ and below by $1/M$ because
$H$ is $M$-quasisymmetric. So, from Lemma~\ref{sd},
$$
k-\zeta(M)
 \leq \frac{H(x+ky) -H(x)}{H(x+y)-H(x)}
 \leq k+\zeta(M).
$$
Thus
$$
(k-\zeta(M))\rho_{H} (x,y) \leq \frac{H(x+ky) -H(x)}{H(x)-H(x-y)}
\leq (k+\zeta (M))\rho_{H}(x,y).
$$
Since $ 1/M \leq \rho_{H} (x,y) \leq M $ and we are assuming that
$0<k \leq 1,$ this implies that
$$
 |\rho_{H} (x,y,k)-k| \leq  \vartheta (M)=M-1+M\zeta (M).
$$
Similarly, we have that
$$
 |\rho_{H} (x,-y,k)-k| \leq  \vartheta (M)=M-1+M\zeta (M).
$$
\end{proof}

\begin{definition}
A bounded positive function $\epsilon(y)$ defined for
positive values of $t$ is called vanishing if $\epsilon(y)
\rightarrow 0^{+}$ as $y \rightarrow 0^+.$
\end{definition}

\begin{definition}
A quasisymmetric circle homeomorphism $h$ is called symmetric (or
asymptotically affine) if it is quasisymmetric and if there exists
a vanishing function $\varepsilon(y)$ such that
\begin{equation}\label{a}
\frac{1}{1+\varepsilon(y)} \leq \rho_{H}(x,y) \leq
1+\varepsilon(y),
\end{equation}
for all real numbers $x$ and all $y>0.$
\end{definition}

We say asymptotically affine in this definition since the  ratio
in the middle expression in (\ref{a}) is identically equal to one
 if and only if $H$ is affine, that is, if $H(x)= ax +b,$ for some
 $a \neq 0.$

\vspace{10pt}
\begin{lemma}~\label{equivalent}
A quasisymmetric circle homeomorphism is symmetric if, and only
if, there is a vanishing function $\epsilon(y)$ such that
\begin{equation}~\label{sa}
\max \{ |\rho_H(x,y,k)-k|, |\rho_H(x,-y,k)-k|\} \leq \epsilon(y)
\end{equation}
for all real numbers $x$ and $y>0$ and for all $k$ with $0<k\leq
1$.
\end{lemma}

\begin{proof} Since $H$ is symmetric, by picking $y>0$
sufficiently small, we obtain a number $M$ arbitrarily close to
$1,$ such that
$$\frac{1}{M} \leq \rho_H(s,t) \leq M$$
for all numbers $s$ and $t$ for which $s-t, s$ and $s+t$ lie in
the interval $[x-y,x+y].$ By Corollary~\ref{sdc}, this implies
that there is a vanishing function $\varepsilon'(y)$ such that
$$
\max \left\{ |\rho_{H}(x,y,k)-k| ,  |\rho_{H}(x,-y,k)-k| \right\}
\leq \varepsilon'(y),$$ for all real numbers $x,$ all $y>0$ and
all $k$ with $0 \leq k \leq 1.$

Conversely, (\ref{sa}) with $k=1$ implies
$$
|\rho_H(x,y)-1| \leq \varepsilon(y),
$$
and this implies the existence of a vanishing function
$\epsilon'(y)$ for which $$\frac{1}{1+\varepsilon'(y)} \leq
\rho(x,y) \leq 1 + \varepsilon'(y) .$$
\end{proof}

\begin{definition}
A circle endomorphism $f$ of degree $d$ is called uniformly
symmetric or uniformly asymptotically affine (UAA) if all of the
inverse branches of $f^n$, $n=1, 2,\ldots,$ are symmetric
uniformly. More precisely, $f^n$ is UAA if there is a vanishing
function $\epsilon (y)$ such that, for all positive integers $n$
and all real numbers $x$ and $y>0$,
\begin{equation}\label{aa}
\frac{1}{1+\varepsilon (y)} \leq \rho_{F^{-n}} (x,y)=\frac{F^{-n}
(x+y) -F^{-n}(x)}{F^{-n}(x)-F^{-n}(x-y)} \leq 1+\varepsilon (y).
\end{equation}
\end{definition}
We say uniformly since the ratio in the middle expression in
(\ref{aa})  approaches $1$ when $y$ approaches $0$  independently
of the number $n$ of compositions of $F.$

\section{Beurling-Ahlfors Extensions}
By definition, a homeomorphism $G$ from a plane domain $\Omega$
onto another plane domain $G(\Omega)$ is quasiconformal if it is
orientation preserving  and if it has locally integrable
distributional first partial derivatives $G_{\overline{z}}$ and
$G_z$ satisfying the inequality
\begin{equation}\label{qc}
      |G_{\overline{z}}(z)| \leq k |G_z(z)|
\end{equation}for some number $k$ with $0\leq k <1$
and for almost all $z.$  The complex valued Beltrami coefficient
$\mu=\mu_G$ for $G$ is defined by the equation
\begin{equation}\label{be}
G_{\overline{z}}(z) = \mu(z) G_z(z)
\end{equation}
where $||\mu(z)||_{\infty}<1.$  It is standard to call the
quantity
\begin{equation}~\label{dil}
K_z(G)=\frac{|G_z(z)|+|G_{\overline{z}} (z)|}
{|G_z(z)|-|G_{\overline{z}}(z)|}=\ \frac{1+|\mu(z)|}{1-|\mu(z)|}
\end{equation}
the dilatation of $G$ at the point the point $z,$ and to call
$$K(G)={\rm essup}_{z \in \Omega} K_z(G)=
\frac{1+||\mu||_{\infty}}{1-||\mu||_{\infty}}$$ the dilatation of
$G$ on the domain $\Omega.$  Thus, the homeomorphism $G$ is
quasiconformal on $\Omega$ if $K(G)<\infty.$  We will also use the
nonstandard notation $K(G,z)$ for the same fraction that appears
in~(\ref{dil}) without the absolute value signs, that is,
$$
K(G,z)=\frac{G_z(z)+G_{\overline{z}}(z)}{G_z(z)-G_{\overline{z}}(z)}.
$$
Note that $K(G,z)$ is complex valued and $|K(G,z)| \leq K(G).$ In
all of these notations, $\mu_G, K_z(G), K(G)$ and $K(G,z),$ we
omit reference to the mapping $G$ if this is clear from the
context.

Consider all possible extensions  of quasisymmetric self-mappings
$H$ of the real axis to quasiconformal self-mappings $\tilde{H}$
of the upper half plane ${\mathbb H},$ and define $K(H)$ by the
formula
$$K(H)= \inf \{K(\tilde{H}): \tilde{H} {\rm \ extends \ } H \}.$$

From the theory of quasiconformal mappings, if $K(H)=1,$ then $H$
is affine, that is, $H(x)=ax+b, a \neq 0.$ Similarly, $H$ is also
affine if $M=M(H)=1$ in the M-condition (\ref{Mcondition}). Thus,
we may take both $M(H)$ and $K(H)$ as measurements of the extent
to which $H$ fails to be affine.  A well-known result of Beurling
and Ahlfors \cite{BeurlingAhlfors} shows that $M(H)$ and $K(H)$
are simultaneously finite and there are estimates for $M(H)$ in
terms of $K(H)$ and vice-versa. Moreover, $M(H)$ and $K(H)$
simultaneously approach $1.$

The Beurling-Ahlfors extension procedure provides a canonical
extension $\tilde{H}$ of any quasisymmetric homeomorphism $H$ such
that the Beltrami coefficient $\mu$ of $\tilde{H}$ satisfies
$\|\mu\|_{\infty}<1.$  Furthermore, it satisfies the following
well-known theorem \cite{GardinerSullivan2}.

\vspace*{10pt}
\begin{theorem}~\label{wk} The Beurling-Ahlfors extension of a
 quasisymmetric self-mapping $H$
of the real axis has a  Beltrami coefficient $\mu$ with
$|\mu(x+iy)|\leq \eta(y)$ for some vanishing function $\eta(y)$
if, and only if, there is a vanishing function $\epsilon(y)$ such
that $$\frac{1}{1+\epsilon(y)} \leq \rho_H(x,y) \leq
1+\epsilon(y).$$
\end{theorem}

In this paper we require a similar estimate that allows the
comparison between distortion functions of two quasisymmetric
homeomorphisms of the real axis $H_0$ and $H_1$ fixing $0,$ $1$
and $\infty.$   The following theorem which was stated
in~\cite{Cui} would be a corollary of Theorem~\ref{wk} if the
Beurling-Ahlfors extension procedure defined a homomorphism.  That
is, if the Beurling-Ahlfors extension of $H_1$ composed with the
Beurling-Ahlfors extension of $H_2$ were equal to the
Beurling-Ahlfors extension of $H_1 \circ H_2.$ But that is not the
case. However, the result is still true and we will give a
complete proof.

\vspace*{10pt}
\begin{theorem}[Cui~\cite{Cui}]~\label{Cui} Suppose the skew quasisymmetric
distortion functions $\rho_0(x,y,k)$ and $\rho_1(x,y,k)$ of $H_0$
and $H_1$ satisfy the inequality
$$
|\rho_0(x,y,k)-\rho_1(x,y,k)|,\;\; |\rho_0(x,-y,k)-\rho_1(x,-y,
k)| \leq \epsilon(y)
$$
for $x,y>0\in {\mathbb R}$ and $0<k\leq 1$, where $\epsilon(y)$ is
a vanishing function. Suppose furthermore that $\mu_1$ and $\mu_2$
are the Beltrami coefficients of the Beurling-Ahlfors extensions
$\tilde{H}_0$ and $\tilde{H}_1,$ that is,
$$\mu_0(z)=\frac{\tilde{H_0}_{\overline{z}}}{\tilde{H_0}_z}
{\ \rm \ and \ }
\mu_1(z)=\frac{\tilde{H_1}_{\overline{z}}}{\tilde{H_1}_z}.$$ Then
there is a vanishing function $\eta(y)$ depending only on
$\epsilon(y)$ such that
$$|\mu_0(x+iy)-\mu_1(x+iy)| \leq \eta(y).$$

Conversely, given two quasiconformal maps $\tilde{H}_0$ and
$\tilde{H}_1$ preserving the real axis and a vanishing function
$\eta(y)$ such that
$$
|\mu_{0}(z)-\mu_{1}(z)|\leq\eta(y),
$$
then there is a vanishing function $\epsilon(y)$ such that
$$
|\rho_0(x,y,k)-\rho_1(x,y,k)|,\;\; |\rho_0(x,-y,k)-\rho_1(x,-y,
k)| \leq \epsilon(y)
$$
for $x,y>0\in {\mathbb R}$ and $0<k\leq 1$, where $H_{0}$ and
$H_{1}$ are the restrictions of $\tilde{H}_{0}$ and
$\tilde{H}_{1}$ to the real axis.
\end{theorem}

\begin{proof}
We take the following formulas as the definition of the
Beurling-Ahlfors extension:
$$
{\widetilde H} = U+iV,
$$ where
\begin{equation}\label{ba1}
U(x,y)=\frac{1}{2y} \int_{x-y}^{x+y} H(s) ds =
\frac{1}{2}\int_{-1}^1 H(x+ky)dk
\end{equation}
and
\begin{equation}\label{ba2}
V(x,y)=\frac{1}{y} \int_x^{x+y} H(s)ds - \frac{1}{y}
\int_{x-y}^{x} H(s)ds. \end{equation}
 In (\ref{ba1}) and (\ref{ba2}) we have
chosen a normalization slightly different from the one given in
\cite{AhlforsBook1}.  It has the property that the extension of
the identity is the identity and  the extension is affinely
natural, by which we mean that for affine maps $A$ and $B,$
$$ \widetilde{id_{\mathbb R}} = \ id_{\mathbb C}$$ and
$$\widetilde{A \circ H \circ B} = A \circ {\widetilde H} \circ B.$$

Note that
\begin{equation}\label{1}
\int_0^1 \rho(x,y,k) dk =\frac{1}{H(x)-H(x-y)} \left(\frac{1}{y}\int_x^{x+y}
H(s) ds -H(x) \right)
\end{equation}
and
\begin{equation}\label{2}
\int_0^1 \rho(x,-y,k) dk = \frac{1}{H(x+y)-H(x)}
\left(H(x)-\frac{1}{y}\int_{x-y}^x H(s)ds\right).
\end{equation}

Let

\begin{equation}\label{leftright}
\begin{array}{lll}
  L & = & H(x)-H(x-y)\\
  R & = & H(x+y)-H(x)\\
  L' & = &  H(x) - \frac{1}{y}\int_{x-y}^{x}H(s)ds , \\
 R' & = & \frac{1}{y}\int_{x}^{x+y} H(s) ds - H(x).\\
 \end{array}
 \end{equation}

\noindent and let $\rho_+ (x,y)= \int_0^1 \rho(x,y,k) dk$ and
$\rho_-(x,y)= \int_0^1 \rho(x,-y,k) dk$. Let
$\rho(x,y)=\rho_{H}(x,y)$. Then

\vspace{.1in}

 \begin{equation}\label{avgs}
 \begin{array}{lll}
  \rho(x,y) & = & R/L\\
   \rho_+(x,y) & = & R'/L\\
   \rho_-(x,y) & = & L'/R.
   \end{array}
 \end{equation}

\vspace{.1in}

\noindent Notice that  for symmetric homeomorphisms the quantity
$\rho$ approaches $1$ and the two quantities $\rho_+$ and $\rho_-$
approach $1/2$ as $y$ approaches zero.
 The complex dilatation of $\tilde{H}$ is given by
 $$\mu(z)=\frac{K(z)-1}{K(z)+1}$$ where
 $$K(z)=\frac{{\tilde H}_{z} + {\tilde H}_{\overline{z}}}
 {{\tilde H}_{z} - {\tilde H}_{\overline{z}}}=
 \frac{(U+iV)_z + (U+iV)_{\overline{z}}}{(U+iV)_z - (U+iV)_{\overline{z}}}$$
 $$=\frac{(U+iV)_x -i(U+iV)_y + (U+iV)_x +i(U+iV)_y}{(U+iV)_x -
 i(U+iV)_y-
 (U+iV)_x-i(U+iV)_y}$$
 $$=\frac{U_x+iV_x}{V_y-iU_y}.$$
 Thus
 $$K(z)=\frac{1+ia}{b-ic},$$
where $a=V_x/U_x, b=V_y/U_x {\rm \ and \ } c=U_y/U_x.$

 To find estimates for these three ratios we must find expressions
 for the four partial derivatives of $U$ and $V$ in (\ref{ba1}) and
 (\ref{ba2}). In the notation of (\ref{leftright})

\vspace{.1in}

$\begin{array}{lll}
U_x & = & \frac{1}{2y}(R+L),\\
V_x & = & \frac{1}{y}\left(R-L \right),\\
V_y & = & \frac{1}{y}\left(R+L\right)-\frac{1}{y}\left(R'+L'\right) , \\
U_y & = & \frac{1}{2y}\left(R-L\right) -\frac{1}{2y}\left(R'-L'\right) . \\
\end{array}$

\vspace{.05in}

\noindent Thus

\vspace{.05in}

$\begin{array}{cll}
a(1+\rho)& = & 2\frac{R-L}{R+L}\cdot \frac{R+L}{L},\\

b(1+\rho) & = & 2\frac{R+L-R'-L'}{R+L}\cdot \frac{R+L}{L}=
2 \left(R/L + 1 - R'/L  -(R/L)(L'/R)\right),\\

c(1+\rho) & = & \frac{R-L-R'+L'}{R+L} \cdot \frac{R+L}{L}= R/L -1
-R'/L + (R/L)(L'/R).

\end{array}$

\vspace{.1in}

 \noindent Finally, we obtain
\begin{equation}\label{rho}
\begin{array}{cll}
a& = & \frac{2(\rho -1)}{\rho+1}, \\

b & = & \frac{2(\rho+1-\rho_+ -\rho\rho_-)}{\rho+1},\\

c& = & \frac{\rho -1 + \rho_+ +\rho \rho_-}{\rho+1 } .

\end{array}
\end{equation}

Since $K(z)=(1+ia)/(b-ic),$ $K(z)+1=(1+ia+b-ic)/(b-ic),$ we have
$$\mu_1(z)-\mu_0(z) = \frac{K_1(z)-1}{K_1(z)+1}-\frac{K_0(z)-1}{K_0(z)+1} =
2\frac{K_1(z)-K_0(z)}{(K_1(z)+1)(K_0(z)+1)}=$$
$$2\frac{(1+ia_1)(b_0-ic_0)-(1+ia_0)(b_1-ic_1)}{(1+ia_1+b_1-ic_1)
(1+ia_0+b_0-ic_0)}=$$
\begin{equation}\label{main}
2\frac{(a_1-a_0)(i
b_1+c_1)+(b_0-b_1)(1+ia_1)+(c_1-c_0)(i-a_1)}{(1+ia_1+b_1-ic_1)
(1+ia_0+b_0-ic_0)}. \end{equation}

From the equation for $b$ in (\ref{rho}) and the inequalities
$\rho_+<\rho$ and $\rho \rho_- <1,$ we see that $b>0.$  Since this
inequality is true for $b_1$ and for $b_0,$ it follows that the
denominator in (\ref{main}) is greater than $1.$
 These equations show that
if $a_0, b_0, c_0$ converge to $a_1, b_1, c_1,$ as $y$ approaches
zero, then $\mu_0$ approaches $\mu_1.$  Clearly $\rho_0$
approaches $\rho_1$ implies $a_0$ approaches $a_1$.

From the hypothesis
$$
|\rho_0(x,y,k)-\rho_1(x,y,k)|,\;\; |\rho_0(x,-y,k)-\rho_1(x,-y,k)|
\leq \epsilon (y)
$$
for $x,y>0\in {\mathbb R}$ and $0<k \leq 1$, we have that
$$
|\rho_{1+}(x,y)-\rho_{0+}(x,y)| \leq \int_0^1
|\rho_1(x,y,k)-\rho_0(x,y,k)| dk \leq \epsilon (y)
$$
and
$$
|\rho_{1-}(x,y)-\rho_{0-}(x,y)| \leq \int_0^1
|\rho_1(x,-y,k)-\rho_0(x,-y,k)| dk \leq \epsilon (y).
$$
This implies that $b_0, c_0$ converge to $b_1, c_1$, as $y$
approaches zero. This completes the proof of the first half of the
theorem.

Since the subsequent arguments do not require the second half, we
only sketch the proof.  Notice that if $\tilde{H}_0$ and
$\tilde{H}_1$ are quasiconformal self-maps of the complex plane
preserving the real axis with Beltrami coefficients $\mu_0$ and
$\mu_1$ satisfying
$$
|\mu_0(z)-\mu_1(z)| \leq \epsilon(y)
$$
for a vanishing function $\epsilon(y),$ then the quasiconformal
map $\tilde{H}_1 \circ (\tilde{H}_0)^{-1}$ has Beltrami
coefficient $\sigma$ with
$$
|\sigma(z)|\leq \epsilon'(y)
$$
for another vanishing function $\epsilon'(y).$  Then $\tilde{H}_1
\circ (\tilde{H}_0)^{-1}$ carries the extremal length problem for
the family of curves joining $[-\infty, \tilde{H}_0(x-y)]$ to
$[\tilde{H}_0(x),\tilde{H}_0(x+ky)]$ to the extremal length
problem for the family of curves joining $[-\infty,
\tilde{H}_1(x-y)]$ to $[\tilde{H}_1(x),\tilde{H}_1(x+ky)].$ If
$\Lambda_0(x,y,k)$ and $\Lambda_1(x,y,k)$ are these two extremal
lengths, then by the Gr\"otzsch argument there is another
vanishing function $\epsilon''(y)$ such that
$$\left|\log \frac{\Lambda_0(x,y,k)}{\Lambda_1(x,y,k)} \right| \leq \epsilon''(y).$$

In \cite[pages 74-76]{AhlforsBook2} Ahlfors shows that if
$\Lambda$ is the extremal length of the curve family that joins
the interval $[-\infty,-1]$ to $[0,m],$ $\Lambda$ is an increasing
real analytic function of $m.$ In particular,
\begin{equation}\label{CI}
\left|\log \frac{m_0}{m_1}\right|<\epsilon {\rm \ if \ and \ only
\ if \ } \left|\log
\frac{\Lambda_0}{\Lambda_1}\right|<\epsilon'\end{equation} and
$\epsilon$ and $\epsilon'$ approach zero simultaneously.

Hence by~(\ref{CI}) there is another vanishing function $\eta(y)$
such that
$$
\left|\frac{H_0(x+ky)-H_0(x)}{H_0(x)-H_0(x-y)}
-\frac{H_1(x+ky)-H_1(x)}{H_1(x)-H_1(x-y)}\right| \leq \eta(y).
$$
Similarly, we have that
$$
\left|\frac{H_0(x)-H_0(x-ky)}{H_0(x+y)-H_0(x)}
-\frac{H_1(x)-H_1(x-ky)}{H_1(x+y)-H_1(x)}\right| \leq \eta(y).
$$
This completes the proof of the second half of the theorem.
\end{proof}

\section{The UAA \te\ space}
The endomorphism $p(z)=z^m$ of $\cir$ is a degree $m$ circle
endomorphism and its lift via the covering mapping $\pi$ is
$P(x)=m \ x.$  That is, $P(0)=0$ and $\pi \circ P = p \circ \pi.$
Obviously, $p^n$ is UAA with constant $M=1.$  In fact, the
restriction to the unit circle of the ratio of Blaschke products,
$$
f(z)=\frac{\prod_{j=1}^{k+m}\frac{z-\alpha_j}{1-\overline{\alpha_j}z}}
{\prod_{j=1}^{k}\frac{z-\beta_j}{1-\overline{\beta_j}z}},
$$
for sufficiently small $|\alpha_j|$ and $|\beta_j|,$ is  also a
degree $m$ UAA circle endomorphism. The following theorem has been
proved for any one-dimensional Markov maps with bounded geometry.
A UAA map of the circle is a Markov map with bounded geometry. The
reader may refer to~\cite{Jiang0,Jiang1,Jiang} for a detailed
proof. However, for the convenience of the reader, we outline a
proof.

\begin{theorem} \label{qsconjugacy}
Given any degree $m$ UAA circle endomorphism $f,$ there  exists a
unique quasisymmetric map $h$ such  that $h \circ p \circ
h^{-1}=f,$ where $p(z)=z^m.$
\end{theorem}
\begin{proof} We begin by using the dynamics of the iterations  of $p$
and $f$ to construct a self map
$H$ of ${\mathbb R}$ satisfying

 i) $H(0)=0,$

 ii) $H \circ T = T \circ H$  and

 iii) $H \circ M = F \circ H.$

\noindent From $H(0)=0$ and $H \circ T^k(0)= T^k \circ H(0),$ we
conclude that $H(k)=k.$  Note that $F \circ T(0)=T^m \circ F (0)$
and $F(0)=0$ implies that $F(1)= m.$  Also, $F \circ T = T^m \circ
F$ implies
$$F^n \circ T(0) = T^{m^n}(0)$$ and so $F^n(1)=m^n.$
Since $F$ is an increasing homeomorphism, $F(0)=0$ and
$F^n(1)=m^n,$ we may select numbers $a_{j,n}$ between $0$ and $1$
such that  $F^n(a_{j,n})=j$ for integers $j$ and $n$ with $0
<j<m^n.$ Then, by definition, if we put $H(j/m^n)=a_{j,n},$  we
obtain
$$ H \circ P^n (j/m^n) = H(j)= j {\rm \ \  and \ \ } F^n \circ
H(j/m^n) = F^n(a_{j.n})= j.$$ This defines $H$ on a dense set of
the unit interval  with the  property that for points $x$ in the
dense set  $H \circ P (x) = F \circ H(x).$ We extend $H$ to a
dense subset of the interval $[k-1,k]$  by requiring that $H \circ
T^k = T^k \circ H.$ If we put $x=j/m^n$ and $t=1/m^n,$  then we
obtain
$$\frac{1}{M} \leq \frac{a_{j+1,n}- a_{j,n}}{a_{j,n}- a_{j-1,n}} \leq  M.$$

Now let $c$ be any number of the form $j/m^n$ and $t$ any positive
number.  Select $k$ so that $1/m^k \leq t \leq 1/m^{k-1}.$ Then
$$\frac{H(c+t)-H(c)}{H(c)-H(c-t)}
\leq \frac{H(c+1/m^{k-1})-H(c)}{H(c)-H(c-1/m^k)} \leq M',$$ where
$M'=1+M+M^2+ \cdots +M^m,$ which depends only on $m$ and $M.$ The
same type of argument yields the lower bound
$$\frac{1}{M'} \leq \frac{H(c+t)-H(c)}{H(c)-H(c-t)}.$$
Since $H$ is continuous and the set of points $j/m^n$ for variable
integers $j$ and $n$ is dense, we conclude that $H$ is
quasisymmetric.
\end{proof}

\begin{definition}  The \te\ space ${\mathcal T}(m)$ consists of all
UAA circle endomorphisms of degree $m>1$ factored by an
equivalence relation. Two endomorpisms $f_0$ and $f_1$
representing elements of ${\mathcal T}(m)$  are equivalent if, and
only if, there is a symmetric homeomorphism $h$ of $\cir$ such
that $h \circ f_0 \circ h^{-1}=f_1.$ Since the dynamics of the
mappings $T$, $P$ and $F$ uniquely determine the points $a_{j,n},$
the mapping $H$ is unique.
\end{definition}

\section{UAC Endomorphisms}
If $f$ is a UAA circle endomorphism, it is possible that $f$ has a
reflection invariant extension $\tilde{f}$ defined in a small
annulus $r<|z|<1/r$ such that
$$\tilde{f}(1/\overline{z})=1/\overline{\tilde{f}(z)},$$
and such that for every $\epsilon>0$ there exists a possibly
smaller annulus $U=\{z:r'<|z|<1/r'\}$ such that
\begin{equation}\label{uac1}
K_z(\tilde{f}^{-n})<1+\epsilon
\end{equation}
 for all $z$ in $U.$  Here
$K_z(g)$ is the dilatation of $g$ at $z$ and inequality
(\ref{uac1}) is meant to hold for almost every $z$ in $U$ and for
all positive integers $n.$  If such an extension exists, then
$\tilde{f}$ is called a uniformly asymptotically conformal (UAC)
dynamical  system.
\begin{lemma}If $\tilde{f}$ is a UAC degree $m$ map defined in a
neighborhood of $\cir,$  then the restriction $f$ of $\tilde{f}$
to $\cir$ is a degree $m$ UAA circle endomorphism.
\end{lemma}
\begin{lemma} For any degree $m$ UAC map $\tilde{f}$ acting on a
neighborhood of $\cir$ with $\tilde{f}(1)=1,$ there is a unique
lift $\tilde{F}$ to an infinite strip containing ${\mathbb R}$ and
bounded by lines parallel to ${\mathbb R}$ such that
\begin{enumerate}
\item{$\pi \circ \tilde{F} = \tilde{f} \circ \pi,$}
\item{$\tilde{F}(0)=0$} \item{$\tilde{F}\circ T = T^m \circ
\tilde{F},$ and} \item{$\tilde{F}$ preserves the real axis and
$\tilde{F}(\overline{z})=\overline{\tilde{F}(z)}.$}
\end{enumerate}
\end{lemma}

\begin{lemma}\label{lift}
In the notation of the previous lemma, if $\tilde{f}$ is UAC then
$\tilde{F}$ is UAC in the sense that for every $\epsilon>0,$ there
is a $\delta>0$ such that if the absolute value of $y = Im \ z$ is
less than $\delta,$ then
\begin{equation}\label{unif1}
K_z(\tilde{F}^{-n})<1+\epsilon.
\end{equation}

Conversely, if $\tilde{F}$ is UAC in the sense that (\ref{unif1})
is satisfied and $\tilde{F}(T(z))=T^m \circ \tilde{F}(z),$ then
the induced map $\tilde{f}$ satisfying $\pi \circ
\tilde{F}=\tilde{f}\circ \pi$
 is UAC.
\end{lemma}

\begin{theorem}\label{uaatouac} If $f$ is a UAA system acting on the unit circle,
then there exists a UAC system $\tilde{f}$ acting in a
neighborhood of the circle such that the restriction of
$\tilde{f}$ to the circle is equal to $f.$
\end{theorem}

\begin{proof} Let $F$ be the lift to the real axis of
$f$ such that $F(0)=0,$ $F \circ T = T^m \circ F$ and such that
$\pi \circ F = f \circ \pi.$  By Theorem \ref{qsconjugacy} there
is a quasisymmetric homeomorphism $H$ of ${\mathbb R}$ fixing $0$
and $1$ such that

   i) $H \circ P \circ H^{-1} = F$ where $P(x)=m \  x,$  and

   ii) $H \circ T \circ H^{-1} = T$  where $T(x)=x+1.$

\noindent By Lemma \ref{lift} it will suffice to find an extension
$\tilde{F}$ of $F$ such that

   i) $ \tilde{F} \circ T (z) = T^m \circ \tilde{F}(z)$

  and

   ii) the Beltrami coefficients  $\mu_{\tilde{F}^{-n}}$ of
   $\tilde{F}^{-n}$
   satisfy $$|\mu_{\tilde{F}^{-n}}(x+iy)|\leq \epsilon(y)$$
    where $\epsilon(y)$ is independent of $n$ and $x.$

    We define $\tilde{F}$ to be $\tilde{H} \circ M \circ
    \tilde{H}^{-1}.$  Since $\tilde{H}$ extends $H,$ clearly
    $\tilde{F}$ extends $F.$

Suppose $\rho_1(x,y,k)$ and $\rho_2(x,y,k)$ are the skew
quasisymmetric distortions of $F^{-n} \circ H$ and $H.$  By
Corollary~\ref{sdc}, there is a vanishing function $\epsilon(y)$
such that
$$
|\rho_1(x,y,k)-\rho_2(x,y,k)|,\;\; |\rho_1(x,-y, k)-\rho_2(x,-y,
k)| \leq \epsilon(y)
$$
for all real numbers $x,$ all $y>0,$ all $k$ with $0<k\leq 1$ and
all $n\geq 1$. Applying Theorem~\ref{Cui}, there is another
vanishing function $\eta(y)$ such that the Beltrami coefficients
$\mu_{\widetilde{F^{-n} \circ H}}$ and $\mu_{\tilde{H}}$ satisfy
$$
|\mu_{\widetilde{F^{-n} \circ H}}(z)- \mu_{\tilde{H}}(z)| \leq
\eta(y), \quad \forall\; n>0.
$$

Since  $$\widetilde{F^{-n} \circ H}
    = \widetilde{H \circ P^{-n}} = \tilde{H} \circ P^{-n},$$
    we conclude that $$|\mu_{\tilde{H}}(m^{-n}z)-\mu_{\tilde{H}}(z)| \leq
    \eta(y).$$
    Also, since the Beurling-Ahlfors extension is affinely natural,
    $\mu_{\tilde{H}}(T(z))=\mu_{\tilde{H}}(z)$
    and $\tilde{H} \circ T \circ \tilde{H}^{-1}(z) = T(z).$ We
    conclude that $\tilde{F}=\tilde{H} \circ M \circ
    \tilde{H}^{-1}$ and $T$ form a uniformly asymptotically
    conformal circle endomorphism of degree $m$.
\end{proof}


\vspace*{20pt}

\begin{thebibliography}{10}

\bibitem{AhlforsBook1}
L.~V. Ahlfors.
\newblock {\em Lectures on Quasiconformal Mapping}, volume~38 of {\em
  University Lecture Series}.
\newblock Amer. Math. Soc, 2006.

\bibitem{AhlforsBook2}
L.~V. Ahlfors.
\newblock {\em Conformal Invariants: Topics in Geometric Function Theory}.
\newblock McGraw-Hill, New York, 1973.

\bibitem{BeurlingAhlfors}
A.~Beurling and L.~V. Ahlfors.
\newblock The boundary correspondence under quasiconformal mappings.
\newblock {\em Acta Math.}, 96:125--142, 1565.

\bibitem{Cui}
G.~Cui.
\newblock Circle expanding maps and symmetric structures.
\newblock {\em Ergod. Th. \& Dynamical Sys.}, 18:831--842, 1998.

\bibitem{CuiGardinerJiang}
Gardiner~F. Cui, G. and Y.~Jiang.
\newblock Scaling functions for degree 2 cirle endomorphisms.
\newblock {\em Contemp. Math.}, preprint, 2004.

\bibitem{GardinerSullivan2}
F.~P. Gardiner and D.~P. Sullivan.
\newblock Symmetric structures on a closed curve.
\newblock {\em Amer. J. of {M}ath.}, 114:683--736, 1992.

\bibitem{GardinerSullivan}
F.~P. Gardiner and D.~P. Sullivan.
\newblock Lacunary series as quadratic differentials in conformal dynamics.
\newblock {\em Contemporary {M}athematics}, 169:307--330, 1994.

\bibitem{Jiang0}
Y. Jiang.
\newblock Geometry of geometrically finite one-dimensional maps.
\newblock Comm. in Math. Phys., 156 (1993), no. 3, 639-647.

\bibitem{Jiang1}
Y. Jiang.
\newblock Renormalization and Geometry in
One-Dimensional and Complex Dynamics.
\newblock Advanced Series in Nonlinear Dynamics, Vol. 10 (1996),
\newblock World Scientific Publishing Co. Pte. Ltd., River Edge, NJ,
xvi+309 pp. ISBN 981-02-2326-9.

\bibitem{Jiang}
Y.~Jiang.
\newblock {Teichm\"uller structures and dual geometric Gibbs type measure theory for continuous potentials.}
\newblock{\em Preprint, 2006}.

\bibitem{LehtoBook}
O. Lehto.
\newblock {\em Univalent Functions and Teichm\"uller Spaces}.
\newblock Springer-Verlag, New York, Berlin, 1987.

\end{thebibliography}
\end{document}